\documentclass[12pt]{amsart}

\usepackage{amsfonts}
\usepackage{amsmath}
\usepackage{amssymb}
\usepackage[margin=1in]{geometry}
\usepackage{amsthm}    
\usepackage{xcolor}

\usepackage{mathtools}							
\usepackage{commath}							
\usepackage{enumitem}  							
											
\usepackage[overload]{textcase}					
\usepackage{array}								
\usepackage{booktabs}							

\usepackage[english]{babel}						

\usepackage[hidelinks]{hyperref}  					

\makeatletter
\@namedef{subjclassname@2020}{\textup{2020} Mathematics Subject Classification}
\makeatother

\newtheorem{theorem}{Theorem}[section]
\newtheorem{corollary}[theorem]{Corollary}
\newtheorem{lemma}[theorem]{Lemma}
\newtheorem{proposition}[theorem]{Proposition}

\theoremstyle{definition}
\newtheorem{definition}[theorem]{Definition}

\newtheorem{open problem}[theorem]{Open Problem}

\numberwithin{equation}{section}

\newcommand{\N}{\mathbb{N}}
\newcommand{\Z}{\mathbb{Z}}

\newcommand{\R}{\mathbb{R}}
\newcommand{\C}{\mathbb{C}}

\renewcommand{\Re}{\operatorname{Re}}
\renewcommand{\Im}{\operatorname{Im}}

\newcommand{\Mod}{\ \mathrm{mod}\ }

\newcommand{\eps}{\varepsilon}

\newcommand{\I}{\mathrm{i}}
\newcommand{\e}{\mathrm{e}}

\DeclareMathOperator{\sgn}{sgn}
\DeclareMathOperator{\SL}{SL}

\newcommand{\jac}[2]{\biggl( \frac{#1}{#2} \biggr)}

\newcommand{\nth}{n\textsuperscript{th}}
\newcommand{\kth}{k\textsuperscript{th}}

\begin{document}

\title[Pointwise behavior of Riemann's function]{The pointwise behavior of Riemann's function}

\author[F. Broucke]{Frederik Broucke}
\thanks{F. Broucke was partly supported by the Ghent University BOF-grant 01J04017 and partly by a postdoctoral fellowship (grant number 12ZZH23N) of the Research Foundation – Flanders.}

\author[J. Vindas]{Jasson Vindas}
\thanks {J. Vindas was partly supported by Ghent University through the BOF-grant 01J04017 and by the Research Foundation--Flanders through the FWO-grant G067621N}

\address{Department of Mathematics: Analysis, Logic and Discrete Mathematics\\ Ghent University\\ Krijgslaan 281\\ 9000 Gent\\ Belgium}
\email{fabrouck.broucke@UGent.be}
\email{jasson.vindas@UGent.be}

\begin{abstract}
We present a new and simple method for the determination of the pointwise H\"{o}lder exponent of Riemann's function $\sum_{n=1}^{\infty} n^{-2}\sin(\pi n^{2} x)$ at every point of the real line. In contrast to earlier approaches, where wavelet analysis and the theta modular group
were needed for the analysis of irrational points, our method is direct and elementary, being only based on the following tools from number theory and complex analysis: the evaluation of quadratic Gauss sums, the Poisson summation formula, and Cauchy's theorem.
\end{abstract}


\maketitle

\section{Introduction}

According to an account of Weierstrass, Riemann would have suggested the function
\begin{equation} 
\label{Riemann's function}
 f(x) = \sum_{n=1}^{\infty}\frac{\sin(n^{2}\pi x)}{n^{2}} 
\end{equation}
as an example of a function which is continuous but nowhere differentiable. 
In 1916, 
Hardy \cite{Hardy} proved, based on earlier work by him and Littlewood \cite{HardyLittlewood}, that Riemann's function $f$ is not differentiable in a certain subset of $\R$ that contains every irrational point. This seemed to confirm the nowhere differentiability conjecture, but, on the contrary, Gerver \cite{Gerver} showed in 1970 that $f$ is actually differentiable at any rational number of the form $(2r+1)/(2s+1)$, $r,s\in\mathbb{Z}$. His results \cite{Gerver,Gerver1971} in combination with Hardy's ones imply that Riemann's function is not differentiable at any other real number. Gerver's proofs are elementary, but difficult and long. Simpler proofs were found  later  by Smith in 1972 \cite{Smith} and Itatsu in 1981 \cite{Itatsu} (see also \cite{HolschneiderTchamitchian,Mohr}).
They provided more precise information about the pointwise behavior of Riemann's ``non-differentiable'' function, which in particular gives the pointwise H\"{o}lder exponent \cite{JaffardMeyerbook} at any rational point.  This left open the determination of the exact pointwise regularity of Riemann's function at the irrationals.

Duistermaat \cite{Duistermaat} used a variant of Itatsu's approach  to exhibit explicit dependence of the $O$-constants on the analyzed rational point in the Smith--Itatsu asymptotic formulas. His error terms were strong enough to enable him to find an upper bound for the pointwise H\"{o}lder exponent at every irrational point. His upper bound depends on approximation properties of the irrational number by certain continued fractions. The problem of finding the pointwise H\"{o}lderian regularity of Riemann's function at irrational points was finally solved by Jaffard \cite{Jaffard} in 1996, who showed that Duistermaat's upper bound was sharp, that is, it is exactly equal to the pointwise H\"{o}lder exponent. Jaffard's proof is indirect and non-elementary.

A generalization of Jaffard's result to fractional integrals of modular forms has recently been obtained by Pastor \cite[Theorem 2.3(2)]{Pastor}. As Jaffard's, his proof is of Tauberian nature and makes use of wavelet analysis and the action of subgroups of the modular group.

We refer to the articles \cite{Banica-Vega,B-E-DR,Eceizabarrena20,Eceizabarrena21} for other recently revealed fascinating properties of Riemann's function.

The purpose of this paper is to provide a new and self-contained approach for the determination of the pointwise H\"{o}lder exponent of Riemann's function at every point. Our arguments are direct and lead to completely elementary and fairly short proofs that only rely on the following tools: 
 the evaluation of quadratic Gauss sums, the Poisson summation formula, and Cauchy's integral theorem. We highlight that the main novelty in our treatment is the use of a simple contour integral formula for the pointwise analysis at \emph{irrational} numbers. To the best of our knowledge, this is the first explicit proof in the literature that avoids passing through the wavelet transform for the computation of the pointwise H\"{o}lder exponent at each irrational point.

Our method can be sketched as follows. For the sake of convenience, we work with a rescaled and complex version of Riemann's function, namely,
\begin{equation}
\label{eq: Riemann complex function}
	\phi(z) = \sum_{n=1}^{\infty}\frac{1}{2\pi\I n^{2}}e(n^{2}z),
	\end{equation}
where we use the notation $e(z)$ for $\e^{2\pi\I z}$ and  $z = x+\I y$ with $y\geq0$.
 The pointwise properties of Riemann's original function can easily be deduced from those of $\phi$. We are interested in the computation of the pointwise H\"older exponent
\begin{equation}
\label{R eq: holder exponent}
\alpha(x) = \sup\{\alpha>0 \mid \phi(x+h) = P_{x}(h) + O_{x}(\, \abs{h}^{\alpha}) \mbox{ for some polynomial $P_x$} \}.
\end{equation}

  Restricting the complex variable $z$ to the upper half-plane, one has  
\begin{equation*}
\phi'(z) = \frac{1}{2}(\theta(z)-1),
\end{equation*}
where $\theta$ stands for the Jacobi theta function, namely, $\theta(z) = \sum_{n\in \Z}e(n^{2}z)$. Therefore, for each $x\in\mathbb{R}$, we obtain the basic identity
\begin{equation}
\label{R eq: key}
	\phi(x+h) -\phi(x) + \frac{1}{2}h = \frac{1}{2}\lim_{y\to0^{+}} \int_{\I y}^{h+\I y}\theta(x+z)\dif z,
\end{equation}
a formula that was already employed by Itatsu for $x=0$.

 We will exploit the formula \eqref{R eq: key} for the analysis of both rational and irrational numbers $x$. Itatsu and Duistermaat used \eqref{R eq: key} at $x=0$ and then transformation properties (under the theta modular group) to study all rational points. We take a different path, in the spirit of Smith, and use the Poisson summation formula to study the boundary behavior of $\theta(x+z)$. This directly gives an exact expression for the limit of the integral in \eqref{R eq: key} when $x$ is rational that yields an asymptotic series and that we shall discuss in Section \ref{sec: behavior at rational points}. Approximating $x$ by  the $\nth$ convergent $r_n=p_n/q_n$  in its continued fraction expansion when $x$ is irrational and using our exact formula for $\theta(r_n+z)$, one generates sufficiently good bounds for $\theta(x+z)$. The next key step in our method is to use Cauchy's theorem to transform 
\eqref{R eq: key} into

\begin{equation}
\label{R eq: key Cauchy}
	\phi(x+h) -\phi(x) + \frac{1}{2}h = - \frac{1}{2}\int_{\Gamma}\theta(x+z)\dif z,
\end{equation}
where $\Gamma$ is the part of the counterclockwise oriented boundary of the rectangle with vertices $0,$ $h,$ $\I|h|$, and $h+\I|h|$ that lies in the (open) upper half-plane. In Section \ref{sec: behavior at irrational points} we shall combine the crucial formula \eqref{R eq: key Cauchy} with our bounds for $\theta(x+z)$ to give a lower bound for $\alpha(x)$, and hence to obtain a new and simpler proof of Jaffard's theorem.

We would like to point out that a straightforward modification of our arguments from Section \ref{sec: behavior at irrational points} can also be used to deduce Pastor's result \cite[Theorem~2.3(2)]{Pastor}  without having to resort on Tauberian theorems for the wavelet transform. We briefly sketch this in Section \ref{appendix}.

We close this introduction by mentioning other generalizations of the function $\phi$. Its fractional integrals 
$$
\phi_{a}(x)=\sum_{n=1}^{\infty} \frac{e(n^2x)}{n^{a}}
$$
are covered by the results of Jaffard \cite{Jaffard} and Pastor only when $a>1$. For $a\leq 1$, $\phi_a$ is no longer a continuous function. Interestingly, Seuret and Ubis \cite{Seuret-Ubis} have performed a pointwise regularity analysis of the Fourier series $\phi_{a}$ for values of the parameter $a\in (1/2,1]$ in terms of $L^2$-type local H\"{o}lder exponents. Another important generalization of $\phi$ is obtained when one replaces
the frequency $n^{2}$ by a polynomial $P(n)$ and considers
\[
	\phi_{P, a}(x) = \sum_{n=1}^{\infty}\frac{e(P(n)x)}{n^{a}}, \quad a > 1.
\]
When the degree $k$ of $P$ is larger than $2$, the analysis becomes much more difficult: one loses the underlying modularity and Poisson summation only seems to yield good control in intervals $\abs{x-p/q} \ll q^{-k}$, which appear to be too small to obtain the pointwise 
regularity at the irrationals. Only partial results are known, but more can be said about the local
behavior ``on average'' by considering the spectrum of singularities $d_{\phi_{P,a}}(\alpha)$, which is defined as the Hausdorff dimension of the set of points $x$ at which the H\"older exponent of $\phi_{P,a}$ equals $\alpha$. We refer to \cite{ChamizoUbis2007,ChamizoUbis2014} for investigations in this direction.


\section{Preliminaries: quadratic Gauss sums}
\label{sec: preliminaries}

The following exponential sums naturally arise in the analysis of $\phi$ at rational points.
\begin{definition}
Let $q, p, m$ be integers with $(p,q)=1$. The quadratic Gauss sum $S(q,p)$ and the generalized quadratic Gauss sum $S(q,p,m)$ are defined as
\[
S(q,p) = \sum_{j=1}^{q}e\biggl( \frac{pj^{2}}{q} \biggr) \quad \mbox{and} \quad
S(q,p,m) = \sum_{j=1}^{q}e\biggl( \frac{pj^{2}+mj}{q} \biggr).
\]
\end{definition}

The quadratic Gauss sums were already evaluated by Gauss (see e.g.\ \cite[Section 9.10]{Apostol} or \cite[Section~9.3]{MontgomeryVaughan}):
\begin{theorem}
\label{th: evaluation S(q,p)}
Suppose $p$ and $q$ are positive integers with $(p,q)=1$. For odd $n$, define 
\[
	\eps_{n} = 
	\begin{dcases*}
	1 	&if $n\equiv 1\Mod4$,\\
	\I  	&if $n\equiv 3\Mod 4$. 
	\end{dcases*}
\]
Then
\[ 
S(q,p)=
	\begin{dcases}
		\eps_{q}\jac{p}{q}\sqrt{q}				&\text{if } q \text{ is odd},\\
		0								&\text{if } q \equiv 2 \Mod 4,\\
		(1+\I)\overline{\eps_{p}} 
		\jac{q}{p}\sqrt{q}	&\text{if } q\equiv 0 \Mod 4.
	\end{dcases} 
\]
Here, $\displaystyle \jac{p}{q}$ is the Jacobi symbol (see \cite{Apostol,MontgomeryVaughan}).
\end{theorem}

The generalized quadratic Gauss sums $S(q,p,m)$ can be related to $S(q,p)$ as follows. Let $p^{\star}$ be the multiplicative inverse of $p \Mod q$. Suppose first that $m \equiv 2m' \Mod q$ for some $m'$. Then we can complete the square to get
\begin{equation}
\label{eq: relate S(q,p,m) to S(q,p) case 1}
S(q,p,m) = \sum_{j=1}^{q}e\biggl(\frac{p(j+p^{\star}m')^{2}}{q}\biggr) e\biggl(-\frac{p^{\star}m'^{2}}{q}\biggr) = e\biggl(-\frac{p^{\star}m'^{2}}{q}\biggr)S(q,p).
\end{equation}
If there is no such $m'$, then $q$ is even and $m$ odd. In this case we have
\[
S(4q,p) 	= \sum_{j=1}^{2q}e\biggl( \frac{p(2j+p^{\star}m)^{2}}{4q}\biggr) + \sum_{j=1}^{2q}e\biggl( \frac{p(2j)^{2}}{4q}\biggr)
		= 2e\biggl(\frac{p^{\star}m^{2}}{4q}\biggr)S(q,p,m) + 2S(q,p), 
\]
since $2j+p^{\star}m$ runs over all odd residues mod $4q$ when $j$ runs over $\{1,\ldots,2q\}$. Therefore, 
\begin{equation}
\label{eq: relate S(q,p,m) to S(q,p) case 2}
S(q,p,m) = 
	\begin{dcases}
		\frac{1}{2}e\biggl(-\frac{p^{\star}m^{2}}{4q}\biggr) S(4q,p)		&\text{if } q\equiv 2 \mod 4,\\
		0												&\text{if } q\equiv 0 \mod 4.
	\end{dcases}
\end{equation}


\section{behavior at rational points}
\label{sec: behavior at rational points}

In this section we deduce an asymptotic expansion for $\phi$ at every  rational number.  We first prove a simple but crucial lemma that describes the behavior of $\theta$ near rationals\footnote{In the language of modular forms, this lemma gives the expansion of $\theta$ at the cusps.}. This lemma will be used again in Section \ref{sec: behavior at irrational points} to derive bounds for $\theta$ near irrational points. For a complex number $z\neq 0$, we define $z^{-1/2}$ via the principal branch of the logarithm continuously extended to the negative real axis from the upper half-plane, i.e.,\ $\arg(z)\in (-\pi,\pi]$. Accordingly, our convention is thus  $t^{1/2}= \I \abs{t}^{1/2}$ for $t<0,$ which simplifies the writing of some formulas below.

\begin{lemma}
\label{lem: theta in rationals}
Suppose $1\leq p \le q$, $(p,q)=1$ and $\Im z >0$. Then
\[
	\theta\biggl( \frac{p}{q} + z \biggr) = \frac{\e^{\pi\I/4}}{q\sqrt{2}}z^{-1/2}\Biggl( S(q,p) + 
		2\sum_{m=1}^{\infty}S(q,p,m)\exp\biggl( -\frac{\I \pi m^{2}}{2q^{2}z}\biggr)\Biggr). 
\]
\end{lemma}
\begin{proof}
Rearranging terms according to their value mod $q$, we write
\[ 
	\theta\biggl( \frac{p}{q} + z \biggr) = \sum_{n\in \Z}e\biggl(\frac{pn^{2}}{q}\biggr)e(n^{2}z) 
	= \sum_{j=1}^{q}e\biggl(\frac{pj^{2}}{q}\biggr)\sum_{n\in j+q\Z}e(n^{2}z) . 
\]
For fixed $z$, the function $f_{z}: \R \to \C$, $f_{z}:t \mapsto e(zt^{2})$ has Fourier transform

\[ 
	\hat{f_{z}}(u) = \int_{-\infty}^{\infty}f_{z}(t)\e^{-2\pi\I ut}\dif t = \frac{\e^{\pi\I/4}}{\sqrt{2}}z^{-1/2}\exp\biggl(-\frac{\I\pi u^{2}}{2 z} \biggr). 
\]
An application of the well-known Poisson summation formula then yields
\begin{align*}
\theta\biggl( \frac{p}{q} + z \biggr) 	&= \frac{\e^{\pi\I/4}}{q\sqrt{2}}z^{-1/2}\sum_{j=1}^{q}e\biggl(\frac{pj^{2}}{q}\biggr)\sum_{m\in \Z}e\biggl(\frac{mj}{q}\biggr)\exp\biggl(-\frac{\I \pi m^{2}}{2q^{2}z}\biggr)\\
							&=\frac{\e^{\pi\I/4}}{q\sqrt{2}}z^{-1/2}\sum_{m\in \Z}S(q,p,m)\exp\biggl(-\frac{\I \pi m^{2}}{2q^{2}z}\biggr)\\
							&=\frac{\e^{\pi\I/4}}{q\sqrt{2}}z^{-1/2}\Biggl( S(q,p) + 2\sum_{m=1}^{\infty}S(q,p,m)\exp\biggl( -\frac{\I \pi m^{2}}{2q^{2}z}\biggr)\Biggr).
\end{align*}
\end{proof}

 Define the ``twisted'' $\phi$-function  
\[ 
\phi_{q,p}(z) = \sum_{m=1}^{\infty}\frac{S(q,p,m)}{2\pi\I m^{2}}e(m^{2}z). 
\]
Lemma \ref{lem: theta in rationals} allows us to give a short proof of the following theorem, essentially due to Smith \cite{Smith} and Itatsu \cite{Itatsu} (cf.\ \cite{Duistermaat,HolschneiderTchamitchian}; a generalization of Theorem \ref{th: behavior phi at rationals} is given in \cite[Corollary~2.4]{ChamizoUbis2007}).
\begin{theorem}
\label{th: behavior phi at rationals}
Let $p$ and $q$ be integers, $q\ge 1$, $(p,q)=1$. Then 
\[
	\phi(p/q + h) = \phi(p/q) + C_{p/q}^{-}\abs{h}_{-}^{1/2} + C_{p/q}^{+}\abs{h}^{1/2}_{+} - h/2 + R_{q,p}(h),
\]
where $C_{p/q}^{\pm}$ are given by 
\begin{equation}
\label{eq: constants h^{1/2}}
C_{p/q}^{-} = \frac{\e^{3\pi\I/4}}{q\sqrt{2}}S(q,p) \quad \mbox{and} \quad C_{p/q}^{+} = \frac{\e^{\pi\I/4}}{q\sqrt{2}}S(q,p),
\end{equation}
and $R_{q,p}(h)$ satisfies the estimate $R_{q,p}(h) \ll q^{3/2}\abs{h}^{3/2}$. Furthermore, $C_{p/q}^{-} = C_{p/q}^{+} = 0$ (and hence $\phi$ is differentiable at $p/q$) if and only if $q \equiv 2 \Mod 4$. 
\end{theorem}
\begin{proof}
Suppose $y>0$. By equation \eqref{R eq: key},
\[ 
	\phi\biggl(\frac{p}{q} + h+\I y\biggr) 	= \phi\biggl(\frac{p}{q} +\I y\biggr) +\frac{1}{2}\int_{\I y}^{h+\I y}\theta\biggl(\frac{p}{q}+\zeta\biggr) \dif \zeta - \frac{1}{2}h. 
\]
Using Lemma \ref{lem: theta in rationals} and integrating by parts, 
\begin{align*}
\int_{\I y}^{h+\I y}\theta\biggl(\frac{p}{q}+\zeta\biggr) \dif \zeta
&= \frac{\e^{\pi\I/4}}{q\sqrt{2}}\Biggl(S(q,p)\Bigl[2\zeta^{1/2}\Bigr]_{\I y}^{h+\I y} + 
		2 \int_{\I y}^{h+\I y}\zeta^{-1/2}(4q^{2}\zeta^{2})\biggl(\phi_{q,p}\biggl(-\frac{1}{4q^{2}\zeta}\biggr)\biggr)' \dif \zeta\Biggr) \\
&= \frac{2\e^{\pi\I/4}}{q\sqrt{2}}\Biggl(S(q,p)\Bigl[\zeta^{1/2}\Bigr]_{\I y}^{h+\I y} + 
		\biggl[4q^{2}\zeta^{3/2}\phi_{q,p}\biggl(-\frac{1}{4q^{2}\zeta}\biggr)\biggr]_{\I y}^{h+\I y}\\
&\phantom{=\frac{2\e^{\pi\I/4}}{q\sqrt{2}}}\qquad  {} - 6q^{2} \int_{\I y}^{h+ \I y}\zeta^{1/2}\phi_{q,p}\biggl(-\frac{1}{4q^{2}\zeta}\biggr) \dif \zeta \Biggr).
\end{align*}
All the occurring functions have continuous extensions to $\R$. Letting $y\to0^{+}$ we obtain the desired result, with
the constants  $C_{p/q}^{\pm}$ as in 
\eqref{eq: constants h^{1/2}} 
and with 
\begin{equation}
\label{eq: evaluation R_{q,p}}
R_{q,p}(h) = -4q\frac{\e^{-3\pi\I\sgn h/4}}{\sqrt{2}} \phi_{q,p}\biggl(-\frac{1}{4q^{2}h}\biggr)\abs{h}^{3/2} - 6q\frac{\e^{\pi\I/4}}{\sqrt{2}}\int_{0}^{h}t^{1/2}\phi_{q,p}\biggl(-\frac{1}{4q^{2}t}\biggr)\dif t.
\end{equation}
The fact that the coefficients $C_{p/q}^{\pm}$ are both zero if and only if $q \equiv 2 \Mod 4$ is an immediate consequence of Theorem \ref{th: evaluation S(q,p)}. The bound $R_{q,p}(h) \ll q^{3/2}\abs{h}^{3/2}$ easily follows since $\phi_{q,p} \ll \sqrt{q}$, in view of \eqref{eq: relate S(q,p,m) to S(q,p) case 1}, \eqref{eq: relate S(q,p,m) to S(q,p) case 2}, and Theorem \ref{th: evaluation S(q,p)}. 
\end{proof}

Iterating the integration by parts procedure, we obtain a full asymptotic series for the remainder $R_{q,p}$. Indeed, for any $K\in \N$,
\begin{multline*}
	 R_{q,p}(h) = -\frac{\e^{-3\pi\I\sgn h/4}}{\sqrt{2}}\sum_{k=0}^{K}a_{k}q^{2k+1}\phi_{q,p}^{(-k)}\biggl(-\frac{1}{4q^{2}h}\biggr)\e^{k\pi\I(1-\sgn h)/2}\abs{h}^{k+3/2} \\
	 - \frac{\e^{\pi\I/4}}{\sqrt{2}}(K+3/2)a_{K}q^{2K+1}\int_{0}^{h}t^{K+1/2}\phi_{q,p}^{(-K)}\biggl(-\frac{1}{4q^{2}t}\biggr)\dif t,
\end{multline*}
where\footnote{For $k=0$, the product equals $1$ in accordance with the empty product convention.} 
\[
	a_{k} = (-1)^{k}4^{k+1}\prod_{j=1}^{k}(j+1/2)
\]
and $\phi_{q,p}^{(-k)}$ stands for the $\kth$-order primitive 
\[
	\phi_{q,p}^{(-k)}(x) = \sum_{m=1}^{\infty}\frac{S(q,p,m)}{(2\pi\I m^{2})^{k+1}}e(m^{2}x).
\]
A similar asymptotic series was obtained by Duistermaat in \cite{Duistermaat}.

Inspecting the $\kth$ term in this asymptotic series, we see that it is of the form $\abs{h}^{3/2+k}g_{k}^{\pm}(\,\abs{h}^{-1})$, where $\pm=\sgn h$ and where the functions $g_{k}^{\pm}$ are  $4q^{2}$-periodic with zero mean and global H\"older regularity $1/2+k$. One readily verifies that $R_{q,p}$ is a so-called \emph{trigonometric chirp} at $0$ of type $(3/2, 1)$ and of regularity $1/2$. The latter refines a theorem of Jaffard and Meyer \cite[Theorem~7.1]{JaffardMeyerbook} for Riemann's function; see \cite[p.~73]{JaffardMeyerbook} for the precise definition of a trigonometric chirp. The prototypical example of a trigonometric chirp at $0$ of type $(\alpha, \beta)$, $\alpha>-1$, $\beta>0$ is the function $\abs{x}^{\alpha}\sin(\,\abs{x}^{-\beta})$. 

Using the explicit expression for $S(q,p)$ given by Theorem \ref{th: evaluation S(q,p)}, we can exhibit the behavior of $\Re \bigl(\phi(p/q+h) - \phi(p/q)\bigr)$ in a precise fashion, which we summarize in Table \ref{tab: behavior Re(phi(p/q+h)-phi(p/q))}. Note that at some rational points the function $\Re \phi$ has a (finite) left (resp.\ right) derivative, but an infinite right (resp.\ left) derivative. The table displays the behavior of the $1$-periodic  function $\Re\bigl(\phi(p/q+h)-\phi(p/q)\bigr)$. By rescaling the function $\Re \phi $ by a factor of $1/2$, we obtain the well known regularity of the $2$-periodic Riemann's function $f$ \eqref{Riemann's function} at rational points.

\begin{corollary}
\label{cor: differentiability Riemann function}
Suppose $r=p/q$ is rational. If $p$ and $q$ are both odd, then $f$ is differentiable at $r$ and its H\"{o}lder exponent at $r$ is $3/2$; otherwise the H\"older exponent of $f$ at $r$ equals $1/2$.
\end{corollary}

\begin{table}
\caption{Behavior of $\Re \bigl(\phi(p/q+h) - \phi(p/q)\bigr)= \sum_{n=1}^{\infty}\frac{\sin(2\pi n^{2}x)}{2\pi n^{2}}$}
\label{tab: behavior Re(phi(p/q+h)-phi(p/q))}
\setlength{\extrarowheight}{5mm}
\begin{center}
\begin{tabular}{>{$}c<{$} >{$}c<{$} >{$\displaystyle}c<{$} >{$\displaystyle}c<{$}}
	\toprule
	q\Mod 4			&p\Mod4				&h<0					&h>0 \\
	\midrule
	1		&\text{any}		&-\jac{p}{q}\frac{1}{2\sqrt{q}}\sqrt{\,\abs{h}} + O_{q}\bigl(\,\abs{h}\bigr)	&\jac{p}{q}\frac{1}{2\sqrt{q}}\sqrt{h} + O_{q}(h)\\
	3		&\text{any}		&-\jac{p}{q}\frac{1}{2\sqrt{q}}\sqrt{\,\abs{h}} + O_{q}\bigl(\,\abs{h}\bigr)	&-\jac{p}{q}\frac{1}{2\sqrt{q}}\sqrt{h} + O_{q}(h)\\
	2		&\text{any}		& -\frac{1}{2}h+O\bigl(q^{3/2}\abs{h}^{3/2}\bigr)					&-\frac{1}{2}h+ O\bigl(q^{3/2}h^{3/2}\bigr) \\
	0 		&1				& -\jac{q}{p}\frac{1}{\sqrt{q}}\sqrt{\,\abs{h}} + O_{q}\bigl(\,\abs{h}\bigr)	&-\frac{1}{2}h + O\bigl(q^{3/2}h^{3/2}\bigr) \\
	0		&3				& -\frac{1}{2}h + O\bigl(q^{3/2}\abs{h}^{3/2}\bigr)				& \jac{q}{p}\frac{1}{\sqrt{q}}\sqrt{h} + O_{q}(h)\\

	\bottomrule
\end{tabular}
\end{center}
\end{table}

\section{Behavior at irrational points}
\label{sec: behavior at irrational points}

We now investigate the behavior of $\phi$ at irrational points $\rho$. Unlike in the rational case, we will not be able to derive an asymptotic formula for $\phi$ near $\rho$. Instead, we will determine the H\"older exponent $\alpha(\rho)$ introduced in \eqref{R eq: holder exponent}.

We need some preparation in order to state the formula for $\alpha(\rho)$. Denote the $\nth$ convergent in the continued fraction expansion of $\rho$ by $r_{n}=p_{n}/q_{n}$, where $(p_{n},q_{n})=1$. The quality of the approximation of $\rho$ by $r_{n}$ is quantified by the number $\tau_{n}$, which is defined via the relation
\[ 
	\abs{\rho-r_{n}} = \biggl(\frac{1}{q_{n}}\biggr)^{\tau_{n}}. 
\] 
Let $(r_{n_{k}})_{k}$ be the subsequence\footnote{Using a basic property of continued fractions (see \eqref{eq: det continued fractions} below), it is readily seen that this is an infinite subsequence.} of convergents $r_{n_{k}}$ with $q_{n_{k}} \not\equiv 2 \Mod 4$, and set 
\begin{equation}
\label{eq: def tau}
\tau(\rho) \coloneqq \limsup_{k\to\infty} \tau_{n_{k}}.
\end{equation}

\begin{theorem}
\label{th: evaluation alpha(rho)}
Let $\rho$ be irrational. The H\"older exponent $\alpha(\rho)$ of $\phi$ at $\rho$ is given by 
\begin{equation}
\label{eq: evaluation alpha(rho)}
	\alpha(\rho) = \frac{1}{2} + \frac{1}{2\tau(\rho)}.
\end{equation}
The same result also holds for the H\"older exponent at $\rho$ of $\Re\phi$ and $\Im\phi$.
\end{theorem}

The rest of this section is devoted to the proof of Theorem \ref{th: evaluation alpha(rho)}, which consists of two parts, namely, establishing the two inequalities $\ge$ and $\le$ in \eqref{eq: evaluation alpha(rho)}. Let us first recall some basic properties of continued fractions (we refer to \cite{Schmidt} for proofs and more advanced properties). 
The continued fractions have the following properties: for every $n\in \N$, $\tau_{n}> 2$, consecutive convergents $r_{n}$ and $r_{n+1}$ lie on different sides of $\rho$, $\abs{\rho-r_{n+1}}<\abs{\rho-r_{n}}$, and 
\begin{equation}
\label{eq: det continued fractions}
p_{n+1}q_{n}-p_{n}q_{n+1} = (-1)^{n}.
\end{equation}
Since $\abs{r_{n}-r_{n+1}} = 1/(q_{n}q_{n+1})$, we have
\[ \abs{\rho-r_{n}} \leq \frac{1}{q_{n}q_{n+1}} \leq 2\abs{\rho-r_{n}}, \]
so 
\begin{equation}
\label{eq: relation q_{n} and q_{n+1}}
\biggl(\frac{1}{q_{n}}\biggr)^{\tau_{n}-1}\leq \frac{1}{q_{n+1}} \leq 2\biggl(\frac{1}{q_{n}}\biggr)^{\tau_{n}-1}.
\end{equation}

Note also that in view of \eqref{eq: det continued fractions}, we have that $q_{n}$ and $q_{n+1}$ are never both congruent to $2$ modulo 4.

\subsection{The lower bound for $\alpha(\rho)$}
\label{subsection lower bound}

 The lower bound for $\alpha(\rho)$ was first found by Jaffard \cite{Jaffard} by means of Tauberian arguments involving the continuous wavelet transform. To estimate the wavelet transform, Jaffard deduced bounds\footnote{Recently, Pastor has extended such bounds to modular forms, see Lemma \ref{lem: Pastor 3.3} below.} for the theta function near the irrational number $\rho$. We will present a simple proof of these bounds, using Lemma \ref{lem: theta in rationals}. Furthermore, we will show how these bounds directly furnish the lower bound for $\alpha(\rho)$, without needing to pass through the wavelet transform.

Comparing the sum with an integral, we immediately obtain the following estimate: \begin{equation}
\label{eq: elementary estimate theta}
	\sum_{n=1}^{\infty}e(n^{2}z) \ll y^{-1/2}
\end{equation}
for $z = x + \I y$ with $y>0$. Here and below $\ll$ stands for Vinogradov's notation.
\begin{proposition}
\label{prop: bound theta near rho}
Let $z= x+\I y$ with $y>0$. For each $\eps>0$ there exists a $\delta = \delta(\rho, \eps)>0$ such that for $\abs{z} < \delta$, $y>0$ the following bound holds:
\begin{equation}
\label{eq: estimate theta near rho}
	\theta(\rho + z) \ll \abs{z}^{\frac{1}{2\tau(\rho)}-\eps-\frac{1}{2}} + y^{-1/2}\abs{z}^{\frac{1}{2\tau(\rho)}-\eps}.
\end{equation}
\end{proposition}

\begin{proof}
We first derive bounds for $\theta$ near a rational $p/q$ with $(p,q)=1$. By Lemma \ref{lem: theta in rationals},
\[ 
	\theta\biggl(\frac{p}{q} +\zeta\biggr) = \frac{e^{\pi\I/4}}{q\sqrt{2}}\zeta^{-1/2}\sum_{m\in\Z}S(q,p,m)\exp\biggl(-\frac{\pi\I}{2q^{2}\zeta}m^{2}\biggr). 
\]
By the results obtained in Section \ref{sec: preliminaries}, we have that $S(q,p,m) \ll \sqrt{q}$. Estimating via \eqref{eq: elementary estimate theta},
\begin{equation}
\label{eq: bound theta near p/q} 
	\theta\biggl(\frac{p}{q} +\zeta\biggr) \ll 
	\frac{\abs{S(q,p)}}{q\abs{\zeta}^{1/2}} + \frac{\sqrt{q}\abs{\zeta}^{1/2}}{(\Im \zeta)^{1/2}}. 
\end{equation}

Let $N$ be such that $n\ge N$ implies $\tau_{n}\le \tau(\rho) + \eps'$ whenever $q_{n} \not\equiv 2 \Mod 4$, and where $\eps'$ is such that $1/(2\tau(\rho) + 2\eps') = 1/(2\tau(\rho)) - \eps$. Set $\delta \coloneqq 2\abs{\rho-r_{N}}$. For $z$ with $\abs{z} \le \delta$, let $n$ be the unique integer larger than $N$ such that $2\abs{\rho - r_{n+1}} < \abs{z} \le 2\abs{\rho - r_{n}}$, and  set $\zeta= z + (\rho - r_{n+1})$. Then
\begin{equation}
\label{eq: zeta comparable to z}
\frac{1}{2}\abs{z} \le \abs{\zeta} \le \frac{3}{2}\abs{z}, \quad \Im \zeta = \Im z = y.
\end{equation}

Suppose first that $q_{n} \not\equiv 2\Mod 4$. We then apply \eqref{eq: bound theta near p/q} with $p=p_{n+1}$, $q=q_{n+1}$. For the second term, we use \eqref{eq: relation q_{n} and q_{n+1}} to see that
\[
	\sqrt{q_{n+1}} \le q_{n}^{\frac{\tau_{n}-1}{2}} = \abs{\rho-r_{n}}^{\frac{1}{2\tau_{n}}-\frac{1}{2}} \le \sqrt{2}\abs{z}^{\frac{1}{2\tau_{n}}-\frac{1}{2}}.
\]
Since $1/(2\tau_{n}) \ge 1/(2\tau(\rho)) - \eps$ (because $q_{n} \not\equiv 2\Mod 4$) and $\abs{\zeta} \asymp \abs{z}$,  this second term is of the desired order. The first term vanishes if $q_{n+1} \equiv 2\Mod 4$, while otherwise we have 
\[ 
	\frac{1}{\sqrt{q_{n+1}}} = \abs{\rho-r_{n+1}}^{\frac{1}{2\tau_{n+1}}} \leq \abs{z}^{\frac{1}{2\tau(\rho)} - \eps}, 
\]
so this first term is also of the desired order.

Suppose now that $q_{n} \equiv 2 \Mod 4$. We then apply \eqref{eq: bound theta near p/q} with $p=p_{n}$, $q=q_{n}$ and get
\begin{align*}
	\theta(\rho+z) 	&= \theta\biggl(\frac{p_{n}}{q_{n}} + \biggl(\zeta + \frac{p_{n+1}}{q_{n+1}} - \frac{p_{n}}{q_{n}}\biggr)\biggr) \ll \frac{\sqrt{q_{n}}}{\sqrt{y}}\abs{\zeta + \frac{(-1)^{n}}{q_{n+1}q_{n}}}^{1/2} \\
				&= \frac{1}{\sqrt{q_{n+1}}{\sqrt{y}}}\abs{\zeta q_{n}q_{n+1} + (-1)^{n}}^{1/2} \ll \frac{1}{\sqrt{q_{n+1}}{\sqrt{y}}}.
\end{align*}
In the first estimate we employed \eqref{eq: det continued fractions}, while in the last estimate we used that $\abs{\zeta} \ll \abs{\rho-r_{n}} \le 1/(q_{n}q_{n+1})$. Since $q_{n} \equiv 2\Mod 4$, we have that $q_{n+1} \not\equiv 2\Mod 4$ (by \eqref{eq: det continued fractions}), so that we can bound $(q_{n+1})^{-1/2}$ like before.
\end{proof}

Using this bound for $\theta$ near $\rho$, we now deduce the lower bound for $\alpha(\rho)$.
Pick $\eps>0$ arbitrarily and use Proposition \ref{prop: bound theta near rho} to find a $\delta>0$ such that the bound \eqref{eq: estimate theta near rho} holds for $\abs{z}<\delta$.
Suppose $\abs{h} < \delta/\sqrt{2}$. We use again \eqref{R eq: key}, so that
\[
	\phi(\rho+h) -\phi(\rho) 
	= -\frac{1}{2}h + \frac{1}{2}\lim_{y\to0^{+}} \int_{\I y}^{h+\I y}\theta(\rho+z)\dif z.
\]
By Cauchy's theorem, the limit of this integral equals
\[
	\int_{0}^{\I \abs[0]{h}}\theta(\rho+z)\dif z + \int_{\I \abs[0]{h}}^{h+\I \abs[0]{h}}\theta(\rho+z)\dif z - \int_{h}^{h+\I \abs[0]{h}}\theta(\rho+z)\dif z \eqqcolon I_{1} + I_{2} + I_{3}.
\] 
Using the bounds \eqref{eq: estimate theta near rho}, we get
\begin{align*}
I_{1}	&\ll \int_{0}^{\abs[0]{h}}y^{-\frac{1}{2}+\frac{1}{2\tau(\rho)}-\eps} \ll \abs{h}^{ \frac{1}{2} + \frac{1}{2\tau(\rho)} - \eps}, \\
I_{2} &\ll \abs{h}^{-\frac{1}{2}+\frac{1}{2\tau(\rho)}-\eps} \cdot \abs{h} = \abs{h}^{ \frac{1}{2} + \frac{1}{2\tau(\rho)} - \eps}, \\
I_{3} &\ll \abs{h}^{-\frac{1}{2}+\frac{1}{2\tau(\rho)}-\eps} \cdot \abs{h} + \abs{h}^{\frac{1}{2\tau(\rho)}-\eps}\int_{0}^{\abs[0]{h}}y^{-1/2}\dif y \ll  \abs{h}^{ \frac{1}{2} + \frac{1}{2\tau(\rho)} - \eps}.
\end{align*}
Since $\eps$ was arbitrary, $\alpha(\rho) \ge 1/2 + 1/(2\tau(\rho))$. A fortiori, this lower bound also holds for the H\"older exponent at $\rho$ of the real and imaginary part of $\phi$.

\subsection{The upper bound for $\alpha(\rho)$}
\label{subsection: upper bound}

An upper bound for the H\"older exponent at $\rho$ can be obtained from the expansion of $\phi$ at rationals, and was first done by Duistermaat \cite[Proposition 5.2]{Duistermaat}. For the sake of being self-contained, we repeat his proof here.

Let $\eps>0$ be arbitrary, and let $(r_{l})_{l}$ be a subsequence of $(r_{n})_{n}$ with the properties that $q_{l}\not \equiv 2\mod 4$ and that $\tau_{l} \ge \tau(\rho)-\eps$. We will construct a sequence of points $(h_{l})_{l}$ such that $h_{l} \to 0$ and $\phi(\rho+h_{l}) -\phi(\rho)$ is bounded from below by a constant multiple of $\abs{h_{l}}^{1/2+1/(2(\tau(\rho)-\eps))}$. We will do this by exploiting the square root behavior of $\phi$ in $r_{l}$. Set
\begin{equation}
\label{eq: lambda}
x_{l} = \lambda\abs{\rho-r_{l}},
\end{equation}
where $\lambda$ is a fixed positive constant, independent of $l$, to be determined later. 
Using Theorem \ref{th: behavior phi at rationals} and Theorem \ref{th: evaluation S(q,p)} we see that
\[  
	\abs{\phi(r_{l}+x_{l}) -\phi(r_{l})} \geq \frac{x_{l}^{1/2}}{\sqrt{2q_{l}}} -\frac{1}{2}x_{l} + O\bigl((q_{l}x_{l})^{3/2}\bigr), 
\]
where the big-$O$ constant is absolute (independent of $q_{l}$). By equation \eqref{eq: lambda}, and using that $\abs{\rho-r_{l}} \leq q_{l}^{-2}$, this is at least
\[ 
	 \frac{\sqrt{\lambda}}{\sqrt{q_{l}}}\abs{\rho-r_{l}}^{1/2}\biggl( \frac{1}{\sqrt{2}} - \frac{1}{2\sqrt{q_{l}}}\sqrt{\lambda} + O(\lambda)\biggr). 
\] 
If we now fix a $\lambda>0$ sufficiently small, then for $q_{l}$ sufficiently large
\[ 
	\abs{\phi(r_{l}+x_{l}) -\phi(r_{l})} \gg \frac{\sqrt{\lambda}}{\sqrt{q_{l}}}\abs{\rho-r_{l}}^{1/2}.
\] 
Using that $q_{l}^{-(\tau(\rho)-\eps)}\geq \abs{\rho-r_{l}}$, we get 
\[ 
	\abs{\phi(r_{l}+x_{l}) -\phi(r_{l})} \gg \abs{\rho -r_{l}}^{\frac{1}{2}+\frac{1}{2(\tau(\rho)-\eps)}} \asymp \abs{x_{l}}^{\frac{1}{2}+\frac{1}{2(\tau(\rho)-\eps)}}.
\]
Finally, since $\abs[0]{\phi(\rho) - \phi(r_{l})}$ and $\abs[0]{\phi(\rho) - \phi(r_{l}+x_{l})}$ are not both smaller than $\abs[0]{\phi(r_{l}+x_{l})-\phi(r_{l})}/2$, we can take $h_{l}=r_{l}-\rho$ or $h_{l}=r_{l}+x_{l}-\rho$ such that $\abs[0]{\phi(\rho) - \phi(\rho+h_{l})}$ is maximal, and we get
\[ 
	\abs{\phi(\rho) - \phi(\rho + h_{l})} \gg \abs{h_{l}}^{\frac{1}{2}+\frac{1}{2(\tau(\rho)-\eps)}}, \quad h_{l}\to 0. 
\]
Since $\eps$ was arbitrary, this shows that $\alpha(\rho) \le 1/2+1/(2\tau(\rho))$.

With a small modification, the above argument shows that the same upper bound also holds for the H\"older exponent at $\rho$ of the real and imaginary part of $\phi$. Indeed, using the same notation as above, we now define $x_{l}$ by setting $\abs{x_{l}} = \lambda\abs{\rho-r_{l}}$ and by choosing the sign of $x_{l}$ so that $r_{l}+x_{l}$ lies on the side where the square root behavior is present (see Table \ref{tab: behavior Re(phi(p/q+h)-phi(p/q))} for $\Re\phi$; for $\Im\phi$ one can make a similar table). 

\section{Behavior at irrationals for fractional integrals of modular forms}
\label{appendix}

As we indicated in the introduction, our ideas can also be applied to simplify the proof of a theorem of Pastor on the pointwise H\"{o}lder exponents of modular forms at the irrationals. 

We adopt here similar notations as in Pastor's paper \cite{Pastor}. We let $g$ be a modular form of weight $k$ for a subgroup $\Gamma$ of finite index of $\SL(2,\Z)$, let\footnote{For simplicity, we assume (without loss of generality) that the modular form is holomorphic at the cusp at $\infty$ and that this cusp has width 1.} 
\[
	g(z) = \sum_{n=0}^{\infty}c_{n}e(nz)
\]
 be its Fourier expansion at $\infty$, and define for $a>0$ its $a$-fractional integral as
\begin{equation*}
 	g_{a}(z) \coloneqq \sum_{n=1}^{\infty}\frac{c_{n}}{n^{a}}e(nz),
\end{equation*}
which converges uniformly for $z=x\in\R$ if $a>k$ in general, and for $a>k/2$ whenever $g$ is a cusp form \cite[Theorem~2.1(2)]{Pastor}. Let $\beta(\rho)$ be the pointwise H\"{o}lder exponent of $g_{a}$ at an irrational number $\rho$. It was shown in \cite{Chamizo2017} that $\beta(\rho)=a-k/2$ if $g$ is a cusp form. The corresponding computation of $\beta(\rho)$ when $g$ is not a cusp form is one of the main results from \cite{Pastor}. In the remainder of this section we assume that $g$ is \emph{not a cusp form} and that $a>k$.

\begin{theorem}[{\cite[Theorem~2.3(2)]{Pastor}}]
\label{lem: Pastor 2.3}
Let $\rho$ be irrational and let\footnote{One always has $\tau(\rho)\ge2$.}
\begin{equation}
\label{eq: tau modular}
 \tau(\rho) \coloneqq \sup\biggl\{\tau: \abs[2]{\rho - \frac{p}{q}} \ll \frac{1}{q^{\tau}} \text{ for infinitely many noncuspidal rationals } \frac{p}{q}\biggr\}.
\end{equation}
Then, the H\"older exponent $\beta(\rho)$ of $g_{a}$ at $\rho$ is given by 
\[
\beta(\rho)= a - k\left(1-\frac{1}{\tau(\rho)}\right).
\]
\end{theorem}

Note that the case of Riemann's (complex) function \eqref{eq: Riemann complex function} corresponds to $g$ being the theta function, which is modular of weight $1/2$ and for which the noncuspidal rationals are precisely those rationals $p/q$ with $q\not\equiv 2 \Mod 4$.

In order to prove Theorem \ref{lem: Pastor 2.3}, Pastor uses a Tauberian argument involving the wavelet transform, in combination with bounds of $g$ near rationals. In this section, we show that the evaluation of the H\"older exponent $\beta(\rho)$ follows directly from these bounds of $g$ by using basic complex analysis.
 The bounds on $g$ are expressed in the following lemma of Pastor, which extends Proposition \ref{prop: bound theta near rho} to general modular forms, and which we shall use as a black box.
\begin{lemma}[{\cite[Lemma 3.3]{Pastor}}]
\label{lem: Pastor 3.3}
Let $\rho$ be irrational and $\tau=\tau(\rho)$ be given by \eqref{eq: tau modular}. For each $\eps>0$, the following bounds hold:
\begin{equation}
\label{eq: bound 1 app}
	g(x+\I y) \ll y^{\frac{k}{\tau}-\eps-k} + y^{-k}\abs{x+\I y -\rho}^{\frac{k}{\tau}-\eps} \qquad\mbox{ for $0<y<1/2$,}
\end{equation}
and 
\begin{equation}
\label{eq: bound 2 app}
g(\rho+\I y)\gg y^{-k+\frac{k}{\tau}+\varepsilon} \qquad\mbox{infinitely often as }y\to0^{+}.
\end{equation}
\end{lemma}

We slightly adapt our method from Section \ref{subsection lower bound} in order to obtain the lower bound for $\beta(\rho)$. The starting point is the following integral representation for $g_{a}$ (\cite[Eq.\ (3.3)]{Pastor}):
\[
	g_{a}(x) = \frac{(2\pi)^{a}}{\I^{a}\Gamma(a)}\int_{x}^{x+\I\infty}(z-x)^{a-1}\bigl(g(z)-c_{0}\bigr)\dif z.
\]
Let $J$ be the integer such that $a-\bigl(1-1/\tau\bigr)k\in (J,J+1]$. Then by taking $\eps$ small enough in \eqref{eq: bound 1 app} from Lemma \ref{lem: Pastor 3.3}, one sees that the integrals\footnote{Note that $a$ is not necessarily an integer. As customary, $\binom{a-1}{j}= (a-1)\dotsm(a-j)/j!$.}
\[
	\frac{(2\pi)^{a}}{\I^{a}\Gamma(a)}\int_{\rho}^{\rho+\I\infty}(-1)^{j}\binom{a-1}{j}(z-\rho)^{a-1-j}\bigl(g(z)-c_{0}\bigr)\dif z \eqqcolon \frac{g_{a}^{(j)}(\rho)}{j!}, \quad j=1,\dotsc, J,
\]
converge absolutely. Let now $h$ be a small number. By Cauchy's theorem we have
\begin{align*}
	&\frac{i^{a}\Gamma(a)}{(2\pi)^{a}} \biggl(g_{a}(\rho+h) - g_{a}(\rho) - \sum_{j=1}^{J}\frac{g_{a}^{(j)}(\rho)}{j!}h^{j}\biggr) \\
		&= \int_{\rho+h}^{\rho+\I h}(z-\rho-h)^{a-1}\bigl(g(z)-c_{0}\bigr)\dif z - \int_{\rho}^{\rho+\I h}\sum_{j=0}^{J}(-1)^{j}\binom{a-1}{j}(z-\rho)^{a-1-j}h^{j}\bigl(g(z)-c_{0}\bigr)\dif z \\
		&+ \int_{\rho+\I h}^{\rho+\I\infty}\biggl((z-\rho-h)^{a-1} - \sum_{j=0}^{J}(-1)^{j}\binom{a-1}{j}(z-\rho)^{a-1-j}h^{j}\biggr)\bigl(g(z)-c_{0}\bigr)\dif z.
\end{align*}
Applying Taylor's theorem and the bounds provided by Lemma \ref{lem: Pastor 3.3} (and the exponential decay of $g(z)-c_{0}$ as $z\to\infty$), one readily sees that this is bounded by a constant times $ \abs{h}^{a-(1-\frac{1}{\tau})k-\eps}$, showing that $\beta(\rho) \ge a-\bigl(1-1/\tau\bigr)k$.

\bigskip
It might also be possible to apply a variant of Duistermaat's method from Section \ref{subsection: upper bound} in combination with an approximate functional equation to deduce the upper bound for $\beta(\rho)$. This would however require to exhibit explicit bounds in terms of the involved parameters of the error term in e.g. \cite[Theorem 2.5]{Pastor}. We rather choose a shorter path and shall show the upper bound via a simple Abelian argument that involves the maximum modulus principle and Pastor's oscillation bound \eqref{eq: bound 2 app}. We now distinguish two cases according to whether $a$ is a positive integer or not.

Assume first that $a$ is a positive integer. Let $\beta$ be such that $g_{a}(\rho+h) = P_{\rho}(h) + O(\,\abs{h}^{\beta})$ for a certain polynomial $P_{\rho}$. We have
\[
	g(z) = c_{0} + \frac{a!}{(2\pi\I)^{a+1}}\oint\frac{g_{a}(\zeta)}{(\zeta-z)^{a+1}}\dif\zeta, \quad \text{for } \Im z > 0,
\]
where the integral is over a counterclockwise oriented circle with centre $z$ in the upper half-plane. Consider now the function $(z-\rho)^{-\beta}\bigl(g_{a}(z)-P_{\rho}(z -\rho)\bigr)$. It is holomorphic for $\Im z>0$, has a continuous extension to $\{z: \Im z \ge 0 \text{ and } z\neq\rho\}$, it is $O(|z-\rho|^{-\beta})$ when $z \to \rho$, and bounded on $\R\setminus\{\rho\}$. Hence, by the Phragm\'en--Lindel\"of theorem\footnote{We apply here the classical Phragm\'en--Lindel\"of principle after a transformation $\zeta = -1/(z-\rho)$ to functions holomorphic in the open sector $0<\arg \zeta<\pi$ with continuous extension to the boundary.}, it is bounded on $\Im z >0$. Let now $z$ be close to $\rho$. Integrating over a circle of radius $\abs{z-\rho}/2$, we get
\[
	g(z) = c_{0} + \frac{a!}{(2\pi\I)^{a+1}}\oint\frac{g_{a}(\zeta) - P_{\rho}(\zeta-\rho)}{(\zeta-z)^{a+1}}\dif\zeta+O(1) \ll 1 + \abs{z-\rho}^{\beta-a}, \quad \text{as } z\to\rho.
\] 
Comparing this to \eqref{eq: bound 2 app} from Lemma \ref{lem: Pastor 3.3}, we see that $\beta\le a -k\bigl(1-1/\tau\bigr)$.

Suppose now that $a$ is not a positive integer, and let $N=\lceil a \rceil$. We have
\[
	g(z) = c_{0} + \frac{1}{\I^{N}(2\pi)^{a}\Gamma(N-a)}\int_{0}^{\infty}t^{N-a}g_{a}^{(N)}(z+\I t)\frac{\dif t}{t}.
\]
If we again assume that $g_{a}(\rho+h) = P_{\rho}(h) + O(\,\abs{h}^{\beta})$, then by the first part we have $g_{a}^{(N)}(z) \ll \abs{z-\rho}^{\beta-N}+1$ as $z\to \rho$. Hence,
\begin{align*}
	g(\rho+\I y) 	&= c_{0} + \frac{1}{\I^{N}(2\pi)^{a}\Gamma(N-a)}\int_{0}^{1}t^{N-a}g_{a}^{(N)}\bigl(\rho+\I(y+t)\bigr)\frac{\dif t}{t} + O(1) \\
				&\ll \int_{0}^{1}t^{N-a}(y+t)^{\beta-N}\frac{\dif t}{t} + 1 \ll y^{\beta-a}, \quad \text{as } y\to 0,
\end{align*}
provided that $\beta<a$, which we may assume without loss of generality. Comparing again with \eqref{eq: bound 2 app}, this yields $\beta\le a -k\bigl(1-1/\tau\bigr)$.


\end{document}